\documentclass[reqno, 12pt, final, oneside]{amsart}
\usepackage{graphicx}
\usepackage{color}
\usepackage{amsopn}
\usepackage{amssymb,amsfonts,amsthm,amsmath,amstext,latexsym,mathrsfs,esint,bm}
\usepackage{mathtools}
\theoremstyle{plain}
\newtheorem*{acknowledgement*}{Acknowledgements}
\theoremstyle{plain}

\newtheorem{theorem}{Theorem}[section]

\newenvironment{extraeq}[1]{\def\nome{{#1}}$$}{\leqno\rm(\nome)$$
\ignorespaces}
\def\beginfrancoextraeqarray#1{\begin{extraeq}{#1}\begin{array}{rcl}}
\def\beginfrancoeqarray#1{%
\begin{equation}\label{eq:#1}\begin{array}{rcl}}
\def\endfrancoextraeqarray{\end{array}\end{extraeq}}
\def\endfrancoeqarray{\end{array}\end{equation}}
\newenvironment{equazionespezzata}[1]
{\beginfrancoextraeqarray{#1}}{\endfrancoextraeqarray}
\newif\ifextraequationmode

\def\inoptionalequation#1{\ifextraequationmode\begin{equazionespezzata}{#1}
\else\beginfrancoeqarray{#1}\fi}
\def\outoptionalequation{\ifextraequationmode\end{equazionespezzata}
\else\endfrancoeqarray\fi\global\extraequationmodefalse}
\DeclarePairedDelimiter{\pdelim}{(}{)}
\DeclarePairedDelimiter{\bdelim}{[}{]}
\DeclarePairedDelimiter{\Bdelim}{\{}{\}}
\DeclarePairedDelimiter{\vdelim}{\lvert}{\rvert}
\DeclarePairedDelimiter{\Vdelim}{\lVert}{\rVert}
\DeclarePairedDelimiter{\adelim}{\langle}{\rangle}

\let\norm\Vdelim

   \def\Rset{\mathbb{R}}
   
   \def\Nset{\mathbb{N}}
   \def\Zset{\mathbb{Z}}

   \def\supp{\mathop{\rm supp}\nolimits}

   \def\exp{\mathop{\rm exp}\nolimits}

   \def\wlim{\mathop{w-\rm{lim}}}

   \newcommand{\beq}{\begin{equation}}
   \newcommand{\eeq}{\end{equation}}
   \newcommand{\ba}{\begin{array}}
   \newcommand{\ea}{\end{array}}

\def\Chi{{\bm{\chi}}}

\begin{document}


\title{A profile decomposition for the limiting Sobolev embedding}

\author[G. Devillanova]{Giuseppe Devillanova}
  \address{Dipartimento di Meccanica, Matematica e Management Politecnico di Bari, via E. Orabona n. 4, 70125 Bari, Italy}
\email{giuseppe.devillanova@poliba.it}

\author[C. Tintarev]{Cyril Tintarev} 
\address{Sankt Olofsgatan 66B, 75330 Uppsala, Sweden}
\email{tammouz@gmail.com}







\begin{abstract}
For many known non-compact embeddings of two Banach
spaces $E\hookrightarrow F$, every bounded sequence  in $E$  has a subsequence that takes form of  a \emph{profile decomposition} - a sum of clearly structured terms with asymptotically disjoint supports plus a remainder that vanishes in the norm of $F$. 
In this note we construct a profile decomposition for arbitrary sequences in the Sobolev space $H^{1,2}(M)$ of a compact Riemannian manifold, relative to the embedding of $H^{1,2}(M)$ into $L^{2^*}(M)$, generalizing the well-known profile decomposition of Struwe \cite[Proposition 2.1]{Struwe} to the case of arbitrary bounded sequences.
\end{abstract}

\maketitle


\section{Introduction}
\label{sec1}
When the embedding of two Banach spaces $E\hookrightarrow F$ is continuous and not compact, the lack of compactness can be manifested by the (behavior in $F$ of the) difference $u_{k}-u$ between the elements of a weakly convergent sequence $(u_{k})_{k\in\Nset}\subset E$ and its weak limit $u$. 
Therefore one may call \emph{defect of compactness} of $(u_k)_{k\in\Nset}$ the (sequences of) differences $u_{k}-u$ taken up to a suitable remainder that vanishes in the norm of $F$.
(Note that, if the embedding is compact and $E$ is reflexive,
the defect of compactness is itself infinitesimal and so it can be identified with zero). For many embeddings there exist
well-structured representations of the defect of compactness, known
as \emph{profile decompositions}. Best studied are profile decompositions
relative to Sobolev embeddings, which are sums of terms with asymptotically
disjoint supports, called \emph{elementary concentrations}
or \emph{bubbles}. Profile decompositions were originally motivated
by studies of concentration phenomena in PDE in the early 1980's by Uhlenbeck,
Brezis, Coron, Nirenberg, Aubin and Lions, and they play a significant
role in the verification process of the convergence of sequences of functions in applied
analysis, particularly when the information available via the classical
concentration-compactness method is not enough detailed. 

Profile decompositions are known to exist when the embedding $E\hookrightarrow F$ is \emph{cocompact} relative to some group $\mathcal{G}$ of
isometries on $E$, see \cite{SoliTi}. 
We recall that an embedding $E\hookrightarrow F$ is called cocompact relative to a group $\mathcal{G}$ of isometries ($\mathcal{G}$-cocompact for short) if any sequence $(u_{k})_{k\in\Nset}\subset E$ such that 
$g_{k}(u_{k})\rightharpoonup 0$
for any sequence of operators $(g_{k})_{k\in\Nset}\subset \mathcal{G}$ turns out to be infinitesimal in the norm of $F$. (An elementary example due to Jaffard \cite{Jaffard}, which is easy to verify, is cocompactness of embedding of $\ell^{\infty}(\Zset)$ into itself relative to the group of shifts
$\mathcal{G}:=\lbrace g_m:= (a_{n})_{n\in\Nset}\mapsto(a_{n+m})_{n\in\Nset}\;|\; m\in \Zset \rbrace$.)
Up to the authors knowledge the first cocompactness result for functional spaces is \cite[Lemma 6]{LIeb} by E. Lieb which expresses (using different terminology than the present note) that the nonhomogeneous Sobolev space $H^{1,p}(\Rset^{N})$ is cocompactly embedded into $L^{q}(\Rset^N)$, when $N>p$ and $q\in(p,p^{*})$ (where $p^{*}=\frac{N p}{N-p}$), relative to the group of shifts
$u\mapsto u(\cdot-y),\; y\in\Rset^{N}$. 
A profile decomposition relative to a group $\mathcal{G}$ of bijective isometries on a Banach space $E$ represents defect of compactness $u_k-u$ as a sum of \emph{elementary
concentrations}, or \emph{bubbles}, namely $\sum_{n\in\Nset\setminus \{0\}}g_{k}^{(n)}w^{(n)}$
 with some $g_{k}^{(n)}\in\mathcal{G}$ and $w^{(n)}\in E$.
Note that in the above sum the index $n=0$ is not allowed since, in the existing literature,
usually $w^{(0)}$ represents the weak-limit $u$ of the sequence and $(g_k^{(0)})_{k\in\Nset}$ is the constant sequence of constant value the identity map of the space. So, by using this convention, we can use defect of compactness to represent the sequence $(u_k)_{k\in\Nset}$ as a sum of 
$\sum_{n\in\Nset}g_{k}^{(n)}w^{(n)}$ and a remainder vanishing in $F$.
In the above sums each of the elements 
$w^{(n)}$ (for $n\geq 1$), called \emph{concentration profiles},
is obtained as the weak-limit (as $k\rightarrow\infty$) of the ``deflated'' sequence
$((g_{k}^{(n)})^{-1}(u_{k}))_{k\in\Nset}$ .

Typical examples of isometries groups $\mathcal{G}$, involved
in profile decompositions, are the above mentioned group of shifts $u \mapsto u(\cdot-y)$
and the rescaling group, which is a product group of shifts and dilations
$u\mapsto t^{r}u(t\cdot)$, $t>0$, where, for instance, when $u$ belongs to the homogeneous Sobolev space $\dot{H}^{s,p}(\Rset^{N})$ ($N/s>p\geq 1$, $s>0$), $r=r(p,s)=\frac{N-ps}{p}$.

Existence of profile decompositions for general bounded sequences
in $\dot{H}^{1,p}(\Rset^{N})$ (relative to the rescaling group) was
proved by Solimini, see \cite[Theorem 2]{Solimini}, and independently, but with a weaker form of remainder, by G\'erard in \cite{Gerard}, with an extension to the case of fractional Sobolev spaces by Jaffard in \cite{Jaffard}. 
Only in \cite{SchinTin}, for the first time, the authors observed that profile decomposition
(and thus concentration phenomena in general) can be understood in
functional-analytic terms, rather than in specific function spaces.
Actually the results in \cite{SchinTin} where extended in \cite{SoliTi} to uniformly
convex Banach spaces with the Opial condition (without the Opial condition
profile decomposition still exists but weak convergence must be replaced by (a less-known) Delta convergence, see \cite{DSTweak}). Finally the result has been extended up to a suitable class of metric spaces, see \cite{DSTmetric} and \cite{D}.
Despite the character of the statement in \cite{SoliTi} is rather general, profile decompositions are still true, for instance, when the space $E$
is not reflexive (e.g. \cite{AT_BV}), or when one only has a semigroup
of isometries (e.g. \cite{ATPisa}), or when the profile decomposition
can be expressed without the explicit use of a group (e.g. Struwe \cite{Struwe}) and so when \cite[Theorem 2.10]{SoliTi} does not apply. 

The present paper generalizes, in the spirit of \cite[Theorem 2]{Solimini}, Struwe's result \cite[Proposition 2.1]{Struwe} (which provides a profile decomposition for Palais-Smale sequences of particular functionals) to the case of general bounded sequences in $\dot H^{1,2}(M)$, where $M$ is a smooth compact manifold in dimension $N\geq 3$.

The paper is organized as follows. In Section \ref{sec2} we introduce some notation and state the main theorem of the paper and the result on which the related proof is based. In Section \ref{sec3} we prove that the embedding $H^{1,2}(M)\hookrightarrow L^{2^\ast}(M)$ is cocompact with respect to a group of suitable transformations which are depending on the Atlas associated to the manifold. Section \ref{sec4} is devoted to the proof of (the main) Theorem \ref{thm:StSol}.

\section{Statement of the main result}\label{sec2}
Let $N\ge 3$ and let $(M,g)$ be a compact smooth Riemannian $N$-dimensional manifold. We consider the Sobolev space $H^{1,2}(M)$ equipped with the norm defined by the quadratic form of the Laplace-Beltrami operator, 
\begin{equation}
\label{eq:norm}
\|u\|^2=\int_M (|{\mathrm d}u|^2+u^2){\mathrm d}v_g,
\end{equation}
($v_g$ denotes the Riemannian measure of the manifold).
For every $y\in M$ we shall denote by $T_y(M)$ the tangent space in $y$ to $M$, and by $\exp_y$
the exponential (local) map at the point $y$ (defined on a suitable set $U_y\subset T_y(M)$ by setting, for all $v\in U_y$, $\exp_y(v):=\gamma_v(1)$ where $\gamma_v$ is the unique geodesic, contained in $M$, such that $\gamma_v(0)=y$ and $\gamma_v^\prime(0)=v$ and extended to the case $v=0$ by setting $\exp_y(0)=y$).
Since we will not use here any property of tangent bundles we will identify tangent spaces of $M$ at different points with $\Rset^N$ and, for any $\rho>0$, we shall denote by
$B_\rho(0)$ the Euclidean $N$-dimensional ball centered at the origin with radius $\rho$.
On the other hand, we shall denote by $\mathcal{B}_\rho(y)$ the open coordinate ball (i.e.\ the subset in $M$ such that $\exp_y^{-1}(\mathcal{B}_\rho(y))=B_\rho(0)$) with center $y$ and radius $\rho>0$.
For the reader's convenience we recall that the injectivity radius $\rho_y$ of a point $y\in M$ is the radius of the largest ball about the origin in $T_y(M)$ that can be mapped diffeomorfically via the map $\exp_y$, and that, the injectivity radius of the mainfold M, $\rho_M:=\inf_{y\in M}\rho_y$.
Since $M$ is compact, $\rho_M$ is strictly positive, so we can fix $0<\rho<\frac{\rho_M}{3}$,
moreover, there exists a finite set of points $(z_{i})_{i\in I}\subset M$ such that
$\pdelim{\mathcal{B}_\rho(z_{i}),\exp^{-1}_{z_{i}}}_{i\in I}$
is a finite smooth atlas of $M$.
%

In what follows we shall fix $\Chi\in C_0^\infty (B_\rho(0))$, $\Chi=1$ on $B_\frac{\rho}{2}(0)$, so that, set for $i\in I$
\begin{equation}
\label{eq:chi_i}
\hat\chi_i:=\hat\chi_{z_{i}}=\Chi\circ\exp_{z_{i}}^{-1}\qquad \mbox{ and }\qquad
\chi_i:=\frac{\hat\chi_i}{\sum_{j\in I}\hat\chi_j},
\end{equation}
$(\chi_i)_{i\in I}$ is a smooth partition of unity on $M$ subordinated to the covering
$(\mathcal{B}_\rho(z_{i}))_{i\in I}$. Then, since
$\|u\circ \exp_{z_{i}}\|_{L^{2^*}(B_\rho(0))}$ is bounded by the $H^{1,2}(B_\rho(0))$-norm of $u\circ \exp_{z_{i}}$, the Sobolev embedding $H^{1,2}(M)\hookrightarrow L^{2^*}(M)$ can be deduced from the corresponding one on the Euclidean space (by the use of the fixed partition of unity $(\chi_i)_{i\in I}$).
In fact, Theorem~\ref{thm:StSol} below will provide a profile decomposition for bounded sequences in $H^{1,2}(M)$.

Finally we recall that the scalar product  associated with \eqref{eq:norm} can be written with help of the partition of unity $(\chi_s)_{s\in I}$ in the following coordinate form:
\begin{equation}
\label{eq:scalarp}
\begin{split}
\langle \Phi,\Psi \rangle& :=
\sum_{s\in I}\int_{B_\rho(0)} \sum_{i,j=1}^N g_{i,j}^{z_{s}}\partial_i((\chi_s\Phi) (\exp_{z_{s}}(\xi)))\partial_j(\Psi (\exp_{z_{s}}(\xi)))\sqrt{\det(g_{i,j}^{z_{s}})}{\mathrm d} \xi+\\
&\quad \sum_{s\in I}\int_{B_\rho(0)} (\chi_s\Phi) (\exp_{z_{s}}(\xi))
\Psi(\exp_{z_{s}}(\xi))\sqrt{\det(g_{i,j}^{z_{s}})}{\mathrm d} \xi.
\end{split}
\end{equation}

Before stating the theorem, we warn the reader that, given a bounded sequence
$(v_k)_{k\in\Nset}\subset H^{1,2}(B_\rho(0))$ and a vanishing sequence of positive numbers
$(t_k)_{k\in\Nset}$, and setting $r=r(2)=\frac{N}{2^\ast}=\frac{N-2}{2}$,
we will say (with a slight abuse on the definition of weak convergence) that the sequence
$(t_k^r v_k(t_k\cdot))_{k\in\Nset}$ weakly converges to $v\in \dot H^{1,2}(\Rset^N)$ if for any $\varphi\in C_0^\infty(\Rset^N)$ such that 
$\supp\varphi\subset B_\rho(0)$
\begin{equation*}
\int \varphi(x)\, t_k^r v_k(t_kx)\,{{\mathrm d}x} \longrightarrow
\int \varphi(x)\, v(x)\,{{\mathrm d}x}\mbox{ as } k\rightarrow\infty.
\end{equation*}

\begin{theorem}\label{thm:StSol}
Let $M$ be a compact smooth Riemannian $N$-dimensional manifold ($N\ge 3$). Let
$\rho\in(0,\frac{\rho_M}{3})$, let $\Chi\in C_0^\infty(B_\rho(0))$, $\Chi=1$ on $B_\frac{\rho}{2}(0)$, and let $(\chi_i)_{i\in I}$, defined by \eqref{eq:chi_i}, be a smooth partition of unity on $M$ subordinated to the covering $(\mathcal{B}_\rho(z_{i}))_{i\in I}$.
Then, given a bounded sequence $(u_k)_{k\in\Nset}$ in $H^{1,2}(M)$ and, with $r=\frac{N}{2^\ast}=\frac{N-2}{2}$, there exist:
\begin{itemize}
\item a sequence $\pdelim{Y^{(n)}}_{n\in\Nset\setminus\Bdelim{0}}$ of sequences 
$Y^{(n)}:=\pdelim{y_{k}^{(n)}}_{k\in\Nset}\subset M$, $y_{k}^{(n)} \to \bar y^{(n)}\in M$, 
\item a sequence $\pdelim{J^{(n)}}_{n\in\Nset\setminus\Bdelim{0}}$ of sequences 
$J^{(n)}:=\pdelim{j_{k}^{(n)}}_{k\in\Nset}\subset \Rset_+$, 
\item a sequence $\pdelim{w^{(n)}}_{n\in\Nset\setminus\Bdelim{0}}$ of functions (profiles) $w^{(n)}\in\dot H^{1,2}(\Rset^N)$,
\end{itemize}
such that, modulo subsequences, 
\begin{equation}\label{eq:jtoinf}
j_{k}^{(n)}\longrightarrow\infty \mbox{ as } k\rightarrow\infty
\quad \forall n\in\Nset \setminus \{0\},
\end{equation}
\begin{equation}\label{eq:expl_orth}
|j_k^{(n)}-j_k^{(m)}|+2^{j_{k}^{(n)}}d(y_{k}^{(n)},y_{k}^{(m)})\to\infty
 \mbox{ whenever } m\neq n,
\end{equation}
\begin{equation}\label{eq:profileM}
2^{-j_{k}^{(n)}r} u_k\circ \exp_{y_k^{(n)}}(2^{-j_{k}^{(n)}}\cdot)
{\rightharpoonup} w^{(n)}\;\mbox{ in } \dot H^{1,2}(\Rset^N) \mbox{ as } k\rightarrow\infty.
\end{equation}
Moreover, setting for all $k\in\Nset$ 
\begin{equation}\label{eq:pdM}
\mathcal{S}_k(x):=\sum_{n\in\Nset\setminus\Bdelim{0}}
2^{j_{k}^{(n)} r}\,
\Chi\circ\exp^{-1}_{y_k^{(n)}}(x)\; w^{(n)}\pdelim{2^{j_{k}^{(n)}}\exp^{-1}_{y_k^{(n)}}(x)},\;x\in M, 
\end{equation}
the series $\mathcal{S}_k\in \dot H^{1,2}(M)$ are unconditionally convergent (with respect to $n$) and 
the sequence $(\mathcal{S}_k)_{k\in\Nset}$ is uniformly convergent (with respect to $k$) in 
$\dot H^{1,2}(M)$, in addition
\begin{equation}\label{eq:vanishM}
u_k-u-\mathcal{S}_k\rightarrow 0 \mbox{ in } L^{2^*}(M)\,.
\end{equation}
Finally the following energy bound holds
\begin{equation}\label{PlancherelM}
\sum_{n\in\Nset\setminus\{0\}}\norm{\nabla w^{(n)}}^2_{L^2(\Rset^N)}+\norm{u}^2_{H^{1,2}(M)}\le \liminf_{k\rightarrow \infty} \norm{u_k}^2_{H^{1,2}(M)}.
\end{equation}
\end{theorem}

We want to emphasize that \eqref{eq:vanishM} states that, modulo subsequence, the defect of compactness $u_k-u$ of the bounded sequence $(u_k)_{k\in\Nset}$ (which, modulo subsequence, weakly converges to $u$) has a representation given (up to a remainder which vanishes in the norm of $L^{2^*}(M))$ by the clearly structured terms in $\mathcal{S}_k$.

The proof of this theorem is based on the following easy corollary to Solimini's profile decomposition \cite[Theorem 2]{Solimini}.

\begin{theorem}\label{thm:PDsob:2} 
Given $m\in\Nset\setminus \lbrace 0 \rbrace$ and $1<p<\frac{N}{m}$ let $r=\frac{N}{p^\ast(m)}=\frac{N-mp}{p}$.
Let $(v_k)_{k\in\Nset}$ be a bounded sequence in the homogeneous Sobolev space
$\dot H^{m,p}(\Rset^N)$ supported on a compact set $K\subset \Rset^N$. Then, there exists a (renamed) subsequence (s.t. $v_k\rightharpoonup v$) whose defect of compactness $v_k-v$ has the form
\begin{equation*}\label{eq:PDCompsup}
S_k=\sum_{n\in\Nset\setminus\Bdelim{0}}2^{j_k^{(n)}r} w^{(n)}(2^{j_k^{(n)}}(\cdot-\xi_k^{(n)})),
\end{equation*}
where, for any $n\in \Nset \setminus \{0\}$,
$\Xi^{(n)}:=(\xi_k^{(n)})_{k\in\Nset}\subset K$, and 
$J^{(n)}:=(j_k^{(n)})_{k\in\Nset}\subset \Rset$ are such that 
$j_k^{(n)}{\rightarrow}+\infty$ as $k\to \infty$ and 
$w^{(n)}$ is the weak limit of the sequence $\pdelim{2^{-j_k^{(n)}r} v_k(2^{-j_k^{(n)}}\cdot+\xi_k^{(n)})}_{k\in\Nset}$. Moreover the addenda are asymptotically mutually orthogonal, i.e.\
\begin{equation}\label{eq:orth0}
|j_k^{(n)}-j_k^{(m)}|+2^{j_k^{(n)}}|\xi_k^{(n)}-\xi_k^{(m)}|\to\infty \mbox{ whenever } m\neq n.
\end{equation}
\end{theorem}

\begin{proof}
We shall assume, without restrictions, that $u_k\rightharpoonup 0$.
According to the profile decomposition result \cite[Theorem 2]{Solimini}, modulo the extraction of a subsequence, each term $v_k$ has concentration terms (depending on $n$) of the following shape 
\begin{equation*}\label{eq:conc_R}
c_{k}^n:=2^{j_k^{(n)}r}w^{(n)}(2^{j_k^{(n)}}(\cdot-\xi_k^{(n)}))
\end{equation*}
for some 
$\xi_k^{(n)}\in\Rset^N$, $j_k^{(n)}\in \Rset$ 
where $w^{(n)}$ is obtained as the weak limit of the sequence $\pdelim{2^{-j_k^{(n)}r} v_k(2^{-j_k^{(n)}}\cdot+\xi_k^{(n)})}_{k\in\Nset}$.
We claim that the sequence $J^{(n)}$ is bounded from below.
Indeed, on the contrary, the assumption $j_k^{(n)}{\rightarrow} -\infty$ as $k\to\infty$ would imply, since $v_k$ has a bounded support, that
\begin{equation*}
\norm{2^{-j_k^{(n)}r} v_k\pdelim{2^{-j_k^{(n)}}\cdot+\xi_k^{(n)}}}_p
{\rightarrow} 0 \mbox{ as } k\to\infty,
\end{equation*}
and so that $w^{(n)}=0$.

As a consequence $\xi_k^{(n)}\in K$ for $k$ large enough.
Note also that $J^{(n)}$ cannot have any bounded subsequence, since otherwise $(v_k)_{k\in\Nset}$ should have a nonzero weak limit, in contradiction to our assumptions.

Finally, condition \eqref{eq:orth0} is the condition of asymptotic orthogonality (decoupling) of bubbles from \cite{Solimini}. 
\end{proof}


\section{Cocompactness in Sobolev spaces of compact manifolds}\label{sec3}
The Sobolev embedding $H^{1,2}(M)\hookrightarrow L^{2^*}(M)$ has the following  property of cocompactness type.
\begin{theorem}
Let $M$ be a compact smooth Riemannian $N$-dimensional manifold ($N\ge 3$), and 
$0<\rho<\frac{\rho_M}{3}$. Let $\pdelim{\mathcal{B}_\rho(z_{i}),\exp^{-1}_{z_{i}}}_{i\in I}$
be a finite smooth atlas of $M$ and let $\Chi\in C_0^\infty (B_\rho(0))$ so that $(\chi_i)_{i\in I}$, defined by \eqref{eq:chi_i}, is a smooth partition of unity on $M$ subordinated to the covering
$(\mathcal{B}_\rho(z_{i}))_{i\in I}$.
Set 
$r=r(2)=\frac{N}{2^\ast}=\frac{N-2}{2}$. If $(u_k)_{k\in\Nset}$ is any bounded sequence in $H^{1,2}(M)$ such that 
for every $i\in I$, $(y_k)_{k\in\Nset}\subset \mathcal{B}_\rho(z_{i})$, and $(j_k)_{k\in\Nset}\subset\Nset$ such that $j_k\rightarrow +\infty$
\begin{equation}\label{eq:Struwevanish}
2^{-j_k r}(\chi_i u_k)\circ\exp_{y_k}(2^{-j_k}\cdot)
{\rightharpoonup} 0 \mbox{ as } k \to\infty,
\end{equation}
then $u_k{\rightarrow} 0$ in $L^{2^\ast}(M)$.
\end{theorem}
\begin{proof}
We claim that for all sequences $(\xi_k)_{k\in\Nset}\subset\Rset^N$ and $(j_k)_{k\in\Nset}\subset \Nset$ such that $j_k\rightarrow +\infty$ and for every $i\in I$ we have
\begin{equation}\label{eq:wc01}
2^{-j_k r}(\chi_iu_k)\circ\exp_{z_{i}}(2^{-j_k}\cdot+\xi_k){\rightharpoonup} 0	\quad \mbox{as }k\to\infty.
\end{equation}
Since \eqref{eq:wc01} is obviously true when $|\xi_k|\ge \rho$, (indeed the terms in \eqref{eq:wc01} are identically zero for $k$ large enough), we shall assume $\xi_k\in B_\rho(0)$ for all $k\in\Nset$.
Given $i\in I$, we set $y_k:=\exp_{z_{i}}(\xi_k)\in M$ and denote by 
$\psi_k$ the transition map between the charts $(\mathcal{B}_\rho(z_{i}),\exp_{z_{i}}^{-1})$ and
$(\mathcal{B}_\rho(y_k),\exp_{y_k}^{-1})$ 
i.e.\ we set $\psi_k:= \exp_{y_k}^{-1}\circ \exp_{z_{i}}$
(so that $\exp_{z_{i}}=\exp_{y_k}\circ \psi_k$ and $\psi_k(\xi_k)=0$).
Therefore, for $k$ large enough, by using Taylor expansion of the first order at $\xi_k$
(where, for a lighter notation, we denote by $\psi_k^\prime(\xi_k)$ the Jacobi matrix of $\psi_k$ at $\xi_k$ 
$(\psi_k^\prime(\xi_k))^{-1}$ its inverse and by $|(\psi_k^\prime(\xi_k))^{-1}|$ the corresponding Jacobian,
and drop the dot symbol for the rows-by-columns product) we get, since $j_k\to +\infty$, that 
\begin{equation}
\label{eq:SVM}
\begin{split}
2^{-j_k r}(\chi_i u_k)(\exp_{z_{i}}(2^{-j_k}\xi+\xi_k)) & =
2^{-j_k r}(\chi_i u_k)(\exp_{y_k}\circ\psi_k)(2^{-j_k}\xi+\xi_k)=\\
&\;
2^{-j_k r}(\chi_i u_k)(\exp_{y_k}(2^{-j_k}
(\psi_k^\prime(\xi_k)+o(1))\xi)).
\end{split}
\end{equation}
(we are using the Landau symbol $o(1)$ to denote any (matrix valued) function uniformly convergent to zero).
In correspondence to any test function $\varphi\in C_0^\infty(\Rset^N)$,
\begin{equation*}\label{eq:approdiff}
\begin{split}
&\int_{B_{2\rho}(0)}\varphi(\xi)2^{-j_k r}\bdelim{(\chi_i u_k)\circ\exp_{z_{i}}(2^{-j_k}\xi+\xi_k)-(\chi_i u_k)\circ\exp_{y_k}(2^{-j_k}
\psi_k^\prime(\xi_k)\xi)}{\mathrm d} \xi =\\
&\int_{B_{2\rho}(0)}\varphi(\xi)2^{-j_k r}\bdelim{(\chi_i u_k)\circ\exp_{y_k}\circ\psi_k(2^{-j_k}\xi+\xi_k)-(\chi_i u_k)\circ\exp_{y_k}(2^{-j_k}\psi_k^\prime(\xi_k)\xi)}{\mathrm d} \xi=\\
&|(\psi_k^\prime(\xi_k))^{-1}|2^{j_k\frac{N+2}{2}} \int_{|\eta|<C2^{-j_k}}\varphi(2^{j_k}(\psi_k^\prime(\xi_k))^{-1}\eta)\times \\
&\qquad \bdelim{(\chi_i u_k)\circ\exp_{y_k}\left(\psi_k((\psi_k^\prime(\xi_k))^{-1}\eta+\xi_k\right)-(\chi_i u_k)\circ\exp_{y}(\eta)}{\mathrm d} \eta=\\
&|(\psi_k^\prime(\xi_k))^{-1}|2^{j_k\frac{N+2}{2}}\int_0^1{\mathrm d} s\int_{|\eta|<C2^{-j_k}}\varphi(2^{j_k}(\psi_k^\prime(\xi_k))^{-1}\eta) \times\\
&\qquad\nabla\pdelim{(\chi_iu_k)\circ\exp_{y_k}
(s\psi_k\pdelim{(\psi_k^\prime(\xi_k))^{-1}\eta+\xi_k)+(1-s)\eta}}
\cdot(\psi_k((\psi_k^\prime(\xi_k))^{-1}\eta+\xi_k)-\eta){\mathrm d} \eta,\\
\end{split}
\end{equation*}
(the second equality holds by integrating with respect to the variable $\eta=2^{-j_k}\psi_k^\prime(\xi_k)\xi$). 
Set, for each $s\in[0,1]$,
$\zeta:= s\psi_k\pdelim{(\psi_k^\prime(\xi_k))^{-1}\eta+\xi_k}+(1-s)\eta$, since for
$\eta\rightarrow 0$, $\zeta=\eta+O(|\eta|^2)$ and since the Jacobian of the transformation is close to $1$ in the domain of integration, the modulus of the last expression is bounded by the following one, which, in turn, can be estimated by Cauchy inequality. So, we have
\begin{align*}
C2^{j_k\frac{N+2}{2}}\int_{|\zeta|<C2^{-j_k}}
\varphi(2^{j_k}(\psi_k^\prime(\xi_k))^{-1}\eta(\zeta))
|\nabla (\chi_i u_k)\circ\exp_{y_k}(\zeta)| |\zeta|^2{\mathrm d}\zeta &\le 
\\
C2^{j_k\frac{N+2}{2}}\|\nabla (\chi_iu_k)\circ\exp_{y_k}\|_2 
\left(\int_{|\zeta|<C2^{-j_k}}
|\varphi(2^{j_k}(\psi_k^\prime(\xi_k))^{-1}\eta(\zeta))|^2
|\zeta|^4{\mathrm d}\zeta
\right)^\frac12 &\le 
\\
C2^{j_k\frac{N+2}{2}}\|u_k\|_{H^{1,2}(M)}\left(\int_{|\xi|<C}
|\varphi(\xi)|^2
2^{-4j_k}|\xi|^42^{-j_k N}{\mathrm d}\xi\right)^\frac12\le C2^{-j_k}
{\longrightarrow} 0.
\end{align*}
Therefore, by taking into account \eqref{eq:SVM}, we deduce that both sequences
$\pdelim{2^{-j_k r}(\chi_i u_k)(\exp_{y_k}(2^{-j_k}\cdot))}_{k\in\Nset}$ and
$\pdelim{2^{-j_k r}(\chi_i u_k)(\exp_{z_{i}}(2^{-j_k}\;\cdot+\xi_k))}_{k\in\Nset}$ have the same weak limit and, since \eqref{eq:Struwevanish} holds true, \eqref{eq:wc01} holds too.


Consequently, from the cocompactness of the embedding
$\dot H^{1,2}(\Rset^N)\hookrightarrow L^{2^*}(\Rset^N)$ (\cite[Theorem 1]{Solimini}), 
 it follows that for every $i\in I$,
\begin{equation*}\label{eq:locva}
(\chi_i u_k)\circ\exp_{z_{i}}{\rightarrow} 0
\quad\mbox{ in } L^{2^*}(\Rset^N) \mbox{ as } k\to\infty,
\end{equation*}
and therefore, since $(\chi_i)_{i\in I}$ is a partition of unity subordinated to the atlas 
$\pdelim{\mathcal{B}_\rho(z_{i}),\exp_{z_{i}}^{-1}}_{i\in I}$,
we deduce that
\begin{align*}
\int_M|u_k|^{2^*}{\mathrm d} v_g =
\int_M\vdelim{\sum_{i\in I}\chi_i u_k}^{2^*}{\mathrm d} v_g \le
C \sum_{i\in I}\int_{\mathcal{B}_\rho(z_{i})}|\chi_i u_k|^{2^*}{\mathrm d} v_g \le
\\
C \sum_{i\in I}\int_{B_\rho(0)}|u_k\circ\exp_{z_{i}}(\xi)|^{2^*}{\mathrm d}\xi \rightarrow 0,
\end{align*} 
which proves the statement of the theorem.
\end{proof}


\section{Proof of Theorem \ref{thm:StSol} (profile decomposition)}\label{sec4}

1. Without loss of generality we may assume (by replacing $u_k$ with $u_k-u$) that $u_k\rightharpoonup 0$. 

Then, setting for all $i\in I$
\begin{equation}
\label{eq:vki}
v_{k,i}:=(\chi_i u_k)\circ\exp_{z_{i}}
\end{equation}
we get that the sequence $(v_{k,i})_{k\in\Nset}$ is bounded in $H^{1,2}_0(B_\rho(0))$ (and weakly converges to zero), and so we can consider a profile decomposition of $(v_{k,i})_{k\in\Nset}$ given by
 Theorem~\ref{thm:PDsob:2} when $m=1$ and $r=\frac{N-2}{2}$. An iterated extraction allows to find a subsequence which has a profile decomposition for every $i\in I$ i.e.\ such that for all $i\in I$ the defect of compactness of $v_{k,i}$ has the following form 
\begin{equation*}\label{eq:PDCompsupi}
S_{k,i}=\sum_{n\in\Nset\setminus\lbrace 0\rbrace}
2^{j_{k,i}^{(n)}r} w_i^{(n)}
\pdelim{2^{j_{k,i}^{(n)}}\pdelim{\cdot-\xi_{k,i}^{(n)}}}=: 
\sum_{n\in\Nset\setminus\lbrace 0\rbrace} c_{k,i}^{(n)}.
\end{equation*}

By taking into account \eqref{eq:vki} we will be able to get concentration terms of $\chi_i u_k$ by composing each concentration term $c_{k,i}^{(n)}$ of $v_{k,i}$
with $\exp_{z_{i}}^{-1}$. More in detail we consider for all $i\in I$ the term, defined on $\mathcal{B}_\rho(z_i)$,
\begin{equation}
\label{eq:conc}
\mathcal{C}_{k,i}^{(n)}:= c_{k,i}^{(n)}\circ \exp_{z_{i}}^{-1} =
2^{j_{k,i}^{(n)} r} w_i^{(n)}\pdelim{2^{j^{(n)}_{k,i}}\pdelim{\exp_{z_{i}}^{-1}(\cdot)-\xi^{(n)}_{k,i}}}.
\end{equation}
Setting 
\begin{equation}\label{eq:changesy}
y_{k,i}^{(n)}:=\exp_{z_{i}}(\xi^{(n)}_{k,i})
\end{equation}
we have that 
\begin{equation*}
\label{eq:concC}
\mathcal{C}_{k,i}^{(n)}=
2^{j_{k,i}^{(n)} r} w_i^{(n)}\pdelim{2^{j_{k,i}^{(n)}}\pdelim{\exp_{z_{i}}^{-1}(\cdot)-\exp_{z_{i}}^{-1}(y_{k,i}^{(n)})}}.
\end{equation*}
Since for all $i\in I$ and $n\in\Nset\setminus\lbrace 0\rbrace$
\begin{equation*}
\label{eq:profilesin}
w_i^{(n)}:=\wlim_{k\rightarrow \infty}
2^{-j_{k,i}^{(n)} r}(\chi_i u_k)\circ \exp_{z_{i}}\pdelim{2^{-j_{k,i}^{(n)}}\;\cdot+\xi_{k,i}^{(n)}},
\end{equation*}
we can see that
$w_i^{(n)}$ ``evaluates'' $\chi_i u_k$ on points belonging to $\mathcal{B}_\rho(z_{i})$ which are mapped by $\exp_{z_{i}}^{-1}$ in subsets of $B_\rho(0)$ which are (for large $k$) concentrated around the points $\xi_{k,i}^{(n)}$. So, due to \eqref{eq:changesy}, it is sufficient to evaluate $w_i^{(n)}$ on points which belong also to $\mathcal{B}_\rho(y_{k,i}^{(n)})$. 
So, setting 
\begin{equation}
\label{eq:Bin}
B_{i,k,n}:=\exp_{y_{k,i}^{(n)}}^{-1}(\mathcal{B}_\rho(y_{k,i}^{(n)})\cap\mathcal{B}_\rho(z_{i}))\subset B_\rho(0),
\end{equation}
we shall consider the transition map between the charts $(\mathcal{B}_\rho(y_{k,i}^{(n)}),\exp_{y_{k,i}^{(n)}}^{-1})$ and $(\mathcal{B}_\rho(z_{i}),\exp_{z_{i}}^{-1})$, i.e.\ the map
\begin{equation}\label{eq:changespsi}
\psi_{i,k,n}:= \exp_{z_{i}}^{-1}\circ\exp_{y_{k,i}^{(n)}}
\end{equation}
defined on $B_{i,k,n}$. Note that
$\psi_{i,k,n}(0)=\xi^{(n)}_{k,i}$, moreover, by setting for any $x\in B_{i,k,n}$
\begin{equation}\label{eq:changeseta}
\eta:= 2^{j_{k,i}^{(n)}}\exp^{-1}_{y_{k,i}^{(n)}}(x),
\end{equation}
we have
$\exp_{z_{i}}^{-1}(x)=\psi_{i,k,n}(2^{-j_{k,i}^{(n)}}\eta)$ for all $x\in B_{i,k,n}$.
Therefore (by using Taylor expansion of the first order of the transition map $\psi_{i,k,n}$ at $0$, where, to use a lighter notation we denote by $\psi_{i,k,n}^\prime(0)$ the Jacobi matrix of $\psi_{i,k,n}$ at zero, 
$(\psi_{i,k,n}^\prime(0))^{-1}$ its inverse
and omit the dot symbol for the rows-by-columns product) we deduce
\begin{equation}
\label{eq:approx}
\begin{split}
& 2^{j^{(n)}_{k,i}}\pdelim{\exp_{z_{i}}^{-1}(x)-\xi^{(n)}_{k,i}} =
2^{j^{(n)}_{k,i}}\pdelim{\psi_{i,k,n}( 2^{-j_{k,i}^{(n)}}\eta)-\xi^{(n)}_{k,i}}=
2^{j^{(n)}_{k,i}}\pdelim{\psi_{i,k,n}( 2^{-j_{k,i}^{(n)}}\eta)-\psi_{i,k,n}(0)}=\\
&\qquad=\psi_{i,k,n}^\prime(0)\eta+O( 2^{-j_{k,i}^{(n)}}\eta^2)=
 2^{j^{(n)}_{k,i}} \psi_{i,k,n}^\prime(0) \exp_{y_{k,i}^{(n)}}^{-1}(x)+O\pdelim{2^{j^{(n)}_{k,i}}\pdelim{\exp^{-1}_{y_{k,i}^{(n)}}(x)}^2}.
\end{split}
\end{equation}
%
Without loss of generality, applying Arzel\`a-Ascoli theorem and passing to a suitable subsequence, we can assume that $\pdelim{\psi_{i,k,n}}_{k\in\Nset}$ converges in the norm of $C^1(\Rset^N)$ as $k\to \infty$ to some function $\psi_{i,n}$. 
We claim that,
under a suitable renaming of the profile $w_i^{(n)}$, namely\
by renaming $w_i^{(n)}(\psi_{i,n}^\prime(0)\;\cdot)$ as 
$w_i^{(n)}$, concentration terms  $\mathcal{C}_{k,i}^{(n)}$ (of $\chi_i u_k$) in \eqref{eq:conc} take the following form:
\begin{equation*}
\label{eq:concM}
\tilde{\mathcal{C}}_{k,i}^{(n)}:=
2^{j_{k,i}^{(n)} r}
w_i^{(n)}\pdelim{2^{j_{k,i}^{(n)}}
\exp_{y_{k,i}^{(n)}}^{-1}(\cdot)}.
\end{equation*}
For this purpose we show that, as $k\to\infty$,
\begin{equation*}
\label{eq:in2i}
\int_{\mathcal{B}_\rho(y_{k,i}^{(n)})\cap\mathcal{B}_\rho(z_{i})}\vdelim{2^{j_{k,i}^{(n)}r} \mathrm{d}
\pdelim{ w_i^{(n)}\pdelim{2^{j^{(n)}_{k,i}}(\exp_{z_{i}}^{-1}(x)-\xi^{(n)}_{k,i})} - 
w_i^{(n)}\pdelim{2^{j^{(n)}_{k,i}}\psi_{i,n}^\prime(0)\exp_{y_{k,i}^{(n)}}^{-1}(x)} }}^2
\mathrm{d}v_g
{\rightarrow} 0.
\end{equation*}

Indeed, the previous relation written under the coordinate map $\exp_{y_{k,i}^{(n)}}$, i.e.\ by setting $\xi= \exp_{y_{k,i}^{(n)}}^{-1}(x)$ becomes (by taking into account \eqref{eq:changespsi} and \eqref{eq:Bin})
\begin{align*}
\int_{B_{i,k,n}}\vdelim{2^{j_{k,i}^{(n)} r} 
\nabla\pdelim{w_i^{(n)}\pdelim{2^{j^{(n)}_{k,i}}(\psi_{i,k,n}(\xi)-\xi^{(n)}_{k,i})}- w_i^{(n)}\pdelim{2^{j^{(n)}_{k,i}}\psi_{i,n}^\prime(0)\xi} }}^2
\mathrm{d}\xi
{\rightarrow} 0  \mbox{ as } k\to\infty,
\end{align*}
and, by taking into account \eqref{eq:changeseta} (and by a null extension to whole of $\Rset^N$ of the involved functions), the claim will follow if, as $k\to\infty$,
%
\begin{align*}
2^{-j_{k,i}^{(n)}\frac{N+2}{2}}
\int_{\Rset^N}
\vdelim{\psi_{i,k,n}^\prime(2^{-j_{k,i}^{(n)}}\eta)
\nabla w_i^{(n)}\pdelim{2^{j^{(n)}_{k,i}}(\psi_{i,k,n}(2^{-j_{k,i}^{(n)}}\eta)-\xi^{(n)}_{k,i})}- 
\psi_{i,n}^\prime(0)\nabla w_i^{(n)}(\psi_{i,n}^\prime(0)\eta)
}^2
\mathrm{d}\eta
{\rightarrow} 0.
\end{align*}
This last convergence easily 
follows by Lebesgue dominated convergence theorem, 
indeed (for all $n$ and for all $i$) 
$\nabla w_i^{(n)}\in L^2(\Rset^N)$, and when $k\to\infty$, we have
$j_{k,i}^{(n)}{\rightarrow} +\infty$, and
(by taking into account that convergence of $\pdelim{\psi_{i,k,n}}_{k\in\Nset}$ and $\pdelim{\psi'_{i,k,n}}_{k\in\Nset}$ to $\psi_{i,n}$ and $\psi'_{i,n}$ respectively is uniform)
the pointwise convergence of 
$\psi_{i,k,n}^\prime (2^{-j^{(n)}_{k,i}}\eta){\rightarrow}\psi_{i,n}^\prime(0)$,
$2^{j^{(n)}_{k,i}}\pdelim{\psi_{i,k,n} (2^{-j^{(n)}_{k,i}}\eta)-\xi_i^{(n)}}
{\rightarrow}\psi_{i,n}^\prime(0)\eta$ (as easily follows by \eqref{eq:approx} and \eqref{eq:changeseta}).

It is easy to see now that the renamed profiles $w_i^{(n)}$ are obtained as pointwise limits (and thus also as weak limits)
\begin{equation}
\label{eq:profilesRen}
w_i^{(n)}(\xi)=\lim_{k\rightarrow \infty}
2^{-j_{k,i}^{(n)} r}(\chi_i u_k)\circ \exp_{y_{k,i}^{(n)}}\pdelim{2^{-j_{k,i}^{(n)}}\xi}, \mbox{ for a.e. } \xi\in \Rset^N.
\end{equation}

\vskip3mm
2. Since each $\overline{\mathcal{B}}_\rho(z_i)\subset \mathcal{B}_{2\rho}(z_i)\subset M$ and $M$ is compact, we may assume that for all $n\in\Nset\setminus\Bdelim{0}$ and for all $i\in I$, there exist, up to subsequences, points of concentration
\begin{equation}
\label{eq:limitpoints}
\overline{y}_i^{(n)}:=\lim_{k\rightarrow \infty} y_{k,i}^{(n)}.
\end{equation}

 In order to achieve the orthogonality relation \eqref{eq:expl_orth} we shall 
introduce the following equivalence relation on the set of sequences in $M\times \Rset$. Namely given 
$\pdelim{y_k, j_k}_{k\in\Nset}$ and $\pdelim{y'_k, j'_k}_{k\in\Nset}$ in $M\times \Zset$ we shall write 
\vskip2mm
\noindent
\begin{equation}
\label{eq:Rel}
\tag{$\mathcal{R}$}
\pdelim{y_k, j_k}_{k\in\Nset}\simeq \pdelim{y'_k, j'_k}_{k\in\Nset}
\; \mbox{ when }\;
\pdelim{|j_k-j_k'|+2^{j_k}d(y_k,y'_k)}_{k\in\Nset} \mbox{ is a bounded sequence}.
\end{equation}
Since the set $I$ is a finite set, the number of sequences 
$\pdelim{y_{k,i}^{(n)}, j_{k,i}^{(n)}}_{k\in\Nset}$
which can be equivalent to a fixed sequence 
$\pdelim{y_{k,\bar{\imath}}^{(\bar{n})}, j_{k,\bar{\imath}}^{(\bar{n})}}_{k\in\Nset}$ is finite.
Therefore we can exploit the unconditional convergence with respect to the indexes $(n)$ of the series $S_{k,i}$ and synchronize them by replacing $\bar{n}$ and all the indexes $m$ in the finite set 
\begin{equation*}
\label{eq:ni}
\mathcal N_{\bar n}:=\Bdelim{m\in \Nset \setminus \Bdelim{0}\;|\; \exists i\in I \mbox{ s.t. } \pdelim{y_{k,i}^{(n)}, j_{k,i}^{(n)}}_{k\in\Nset}\simeq \pdelim{y_{k,\bar{\imath}}^{(\bar{n})}, j_{k,\bar{\imath}}^{(\bar{n})}}_{k\in\Nset}}
\end{equation*}
with, say, the smallest integer in $\mathcal N_{\bar n}$.

Thanks to this synchronization procedure the following property
\begin{equation*}
\label{eq:PR1}
\pdelim{y_{k,i_1}^{(n)}, j_{k,i_1}^{(n)}}_{k\in\Nset}\simeq 
\pdelim{y_{k,i_2}^{(m)}, j_{k,i_2}^{(m)}}_{k\in\Nset} 
\quad \Longleftrightarrow\quad
m=n,
\end{equation*}
holds true for all $i_1,i_2\in I$ and $m,n \in\Nset\setminus \Bdelim{0}$.

Note also that when $\pdelim{y_{k,i_1}^{(n)}, j_{k,i_1}^{(n)}}_{k\in\Nset}\simeq 
\pdelim{y_{k,i_2}^{(n)}, j_{k,i_2}^{(n)}}$,
since $\pdelim{\vdelim{j_{k,i_2}^{(n)}-j_{k,i_1}^{(n)}}}_{k\in\Nset}$ is bounded, we can set, modulo subsequences
\begin{equation*}
\label{eq:j}
j(i_1,i_2,n):=
\lim_{k\rightarrow +\infty} j_{k,i_2}^{(n)}-j_{k,i_1}^{(n)}\in\Rset,
\end{equation*}
so that, by redefining $w_{i_2}^{(n)}(2^{-j(i_1,i_2,n)}\cdot)$ as (the corresponding profile) $w_{i_2}^{(n)}$, we can assume that 
$\pdelim{j_{k,i_2}^{(n)}}_{k\in\Nset}=\pdelim{j_{k,i_1}^{(n)}}_{k\in\Nset}$.
Moreover, since also 
$\pdelim{2^{j_{k,i_1}^{(n)}}d\pdelim{y_{k,i_1}^{(n)},y_{k,i_2}^{(n)}}}_{k\in\Nset}$ is bounded, we get (by \eqref{eq:jtoinf}) that (see \eqref{eq:limitpoints}) 
\begin{equation}
\label{eq:erasing}
\bar{y}_{i_1}^{(n)}=\bar{y}_{i_2}^{(n)} \quad \mbox{ for all } \pdelim{y_{k,i_1}^{(n)}, j_{k,i_1}^{(n)}}_{k\in\Nset}\simeq 
\pdelim{y_{k,i_2}^{(n)}, j_{k,i_2}^{(n)}}_{k\in\Nset}.
\end{equation}

Finally, we show that 
the elementary concentrations terms $C_{k,i}^{(n)}$ do not change (up to a vanishing term) by varying 
$\pdelim{y_{k,i}^{(n)}, j_{k,i}^{(n)}}_{k\in\Nset}$ in the same equivalence class.
Namely the following property holds true
\begin{equation*}
\label{eq:PR2}
\pdelim{y_{k,i_1}^{(n)}, j_{k,i_1}^{(n)}}_{k\in\Nset}\simeq 
\pdelim{y_{k,i_2}^{(n)}, j_{k,i_2}^{(n)}}_{k\in\Nset}
\quad\Rightarrow\quad
\Vdelim{C_{k,i_1}^{(n)}-C_{k,i_2}^{(n)}}\to 0,
\end{equation*}
for all $i_1,i_2\in I$.
Since, as shown above, we can assume, without restrictions, that $\pdelim{j_{k,i_1}^{(n)}}_{k\in\Nset}=\pdelim{j_{k,i_2}^{(n)}}_{k\in\Nset}$ (and we shall denote, to shorten notation, their common value as $\pdelim{j_{k}^{(n)}}_{k\in\Nset}$) it will suffice to prove that,
set $\bar\xi^{(n)}_{k,i_1}=\exp_{z_{i_1}}^{-1}y^{(n)}_{k,i_1}$ and
$\bar\xi^{(n)}_{k,i_2}=\exp_{z_{i_1}}^{-1}y^{(n)}_{k,i_2}$, we have
 \begin{equation}
 \int_{\mathcal{B}_\rho(z_{i_1})}\vdelim{2^{j_{k}^{(n)}r} \mathrm{d}
 	\pdelim{ w_{i_1}^{(n)}\pdelim{2^{j^{(n)}_{k}}(\exp_{z_{i_1}}^{-1}(x)-\bar\xi^{(n)}_{k,i_2})} - 
 	w_{i_1}^{(n)}\pdelim{2^{j^{(n)}_{k}}(\exp_{z_{i_1}}^{-1}(x)-\bar\xi^{(n)}_{k,i_1})}
  }}^2
 \mathrm{d}v_g
 {\rightarrow} 0 \mbox{ as } k\to\infty.
 \end{equation}
 Indeed,
we get, modulo subsequences, that
 \begin{equation*}
 2^{j^{(n)}_{k}}|\bar\xi^{(n)}_{k,i_2}-\bar\xi^{(n)}_{k,i_1}|=
 	2^{j^{(n)}_{k}}|\exp_{z_{i_1}}^{-1}y^{(n)}_{k,i_2}-\exp_{z_{i_1}}^{-1}y^{(n)}_{k,i_1}|=
 	2^{j^{(n)}_{k}}|d(y^{(n)}_{k,i_2},z_{i_1})-(y^{(n)}_{k,i_1},z_{i_1})|\le
 	 	2^{j^{(n)}_{k}}d(y^{(n)}_{k,i_2},y^{(n)}_{k,i_1})\to 0.
\end{equation*}
Then, \eqref{eq:expl_orth} follows directly from \eqref{eq:erasing}.

\vskip3mm
3. Consider now the sum $\sum_{n\in\Nset\setminus \Bdelim{0}}\sum_{i\in I}\tilde{\mathcal{C}}_{k,i}^{(n)}$,
with the sequences $\pdelim{y_{k,i}^{(n)}}_{k\in\Nset}$ and $\pdelim{j_{k,i}^{(n)}}_{k\in\Nset}$, which are synchronized at the Step 2 as $\pdelim{y_{k}^{(n)}}_{k\in\Nset}$ and $\pdelim{j_{k}^{(n)}}_{k\in\Nset}$, while  $y_{k}^{(n)}\to \bar y^{(n)}$ and \eqref{eq:profilesRen} takes the form 
\begin{equation}
\label{eq:profilesRenSync}
w_i^{(n)}(\xi)=\lim_{k\rightarrow \infty}
2^{-j_{k}^{(n)} r}(\chi_i u_k)\circ \exp_{y_{k}^{(n)}}\pdelim{2^{-j_{k}^{(n)}}\xi}, \mbox{ for a.e. } \xi\in \Rset^N.
\end{equation}
Since 
$j_k^{(n)}\to\infty$ implies $\exp_{y_{k}^{(n)}}\pdelim{2^{-j_{k}^{(n)}}\xi}\to  \bar y^{(n)}$ in $M$, we have from  \eqref{eq:profilesRenSync} 
\begin{equation}
\label{eq:profilesRenSync2}
w_i^{(n)}(\xi)=\chi_i(\bar y^{(n)}) \lim_{k\rightarrow \infty} 
2^{-j_{k}^{(n)} r}u_k\circ \exp_{y_{k}^{(n)}}\pdelim{2^{-j_{k}^{(n)}}\xi}, \mbox{ for a.e. } \xi\in \Rset^N,
\end{equation}
taking into account that for each $\xi\in\Rset^N$ the limit is evaluated with $k\ge k(\xi)$ with some $k(\xi)$ sufficiently large.
Set
\begin{equation*}
\label{eq:profilen}
w^{(n)}:=\sum_{i\in I}w_i^{(n)}. 
\end{equation*}
Then relation \eqref{eq:profileM} immediately follows from \eqref{eq:profilesRenSync2},  
$w_i^{(n)}=\chi_i(\bar y^{(n)})w^{(n)}$,  and
since, by Step 1, defect of compactness of $\chi_i u_k$ is a unconditionally convergent series, we have
 \begin{align*} 
\sum_{i\in I}\sum_{n\in\Nset\setminus \Bdelim{0}}\tilde{\mathcal{C}}_{k,i}^{(n)}(x)=
\sum_{n\in\Nset\setminus \Bdelim{0}}\sum_{i\in I}\tilde{\mathcal{C}}_{k,i}^{(n)}(x)=
\\
\sum_{n\in\Nset\setminus \Bdelim{0}}\sum_{i\in I} 2^{j_{k}^{(n)}r}
 w_i^{(n)}\pdelim{2^{j_{k}^{(n)}}\exp_{y_{k}^{(n)}}^{-1}(x)}=
\\
\sum_{n\in\Nset\setminus \Bdelim{0}}w^{(n)}\pdelim{2^{j_{k}^{(n)}}\exp_{y_{k}^{(n)}}^{-1}(x)},\; x\in \mathcal B_\rho(y_k^{(n)})\;,
 \end{align*}
which gives \eqref{eq:pdM}. 
\vskip3mm
4. In order to prove the ``energy'' estimate \eqref{PlancherelM}, assume, without loss of generality, that the sum in  \eqref{eq:pdM} is finite and that all $w^{(n)}$ have compact support, and expand by bilinearity the trivial inequality
$\|u-u_k+\mathcal{S}_k\|_{H^{1,2}(M)}^2\ge 0$. Then, by using the norm \eqref{eq:norm} and the representation  \eqref{eq:scalarp} of the scalar product in $H^{1,2}(M)$,
we have
\begin{equation}\label{eq:preParsevalM}
\begin{split}
0 & \le \|u_k\|^2+\|u\|^2-2\langle u_k,u\rangle +2\langle u-u_k,\mathcal{S}_k\rangle+
\\
\; &\sum_n\|2^{j_k^{(n)}r}\,\Chi\circ\exp^{-1}_{y^{(n)}_k}\,w^{(n)}\pdelim{2^{j^{(n)}_k}\exp^{-1}_{y^{(n)}_k}(\cdot)}\|^2-
\\
\; &\sum_{m\neq n}\adelim{2^{j_k^{(m)}r}\,\Chi\circ\exp^{-1}_{y^{(m)}_k}\,w^{(m)}\pdelim{2^{j^{(m)}_k}\exp^{-1}_{y^{(m)}_k}(\cdot)},
2^{j_k^{(n)} r}\,\Chi\circ\exp^{-1}_{y^{(n)}_k}\,w^{(n)}\pdelim{2^{j^{(n)}_k}\exp^{-1}_{y^{(n)}_k}(\cdot)}}.
\end{split}
\end{equation}
The first line of \eqref{eq:preParsevalM} can be evaluated taking into account that $u_k\rightharpoonup u$, $\mathcal{S}_k\rightharpoonup 0$, that the definition of profiles $w^{(n)}$ given by \eqref{eq:profileM}
 and that $r=\frac{N-2}{2}$.
\begin{align*}
\|u_k\|^2+\|u\|^2-2\langle u_k,u\rangle +2\langle u-u_k,\mathcal{S}_k\rangle=
\nonumber
\\
\|u_k^2\|+\|u^2\|-2\|u\|^2+o(1)-2\sum_n\adelim{ u_k,
2^{j_k^{(n)}r}\,\Chi\circ\exp^{-1}_{y^{(n)}_k}\,w^{(n)}\pdelim{2^{j^{(n)}_k}\exp^{-1}_{y^{(n)}_k}(\cdot)}}=
\nonumber
\\
\|u_k\|^2-\|u\|^2+o(1)-
\nonumber
\\
2\sum_n2^{j_k^{(n)}r}\int_{|\xi|<\rho}\sum_{i,j=1}^N g_{ij}^{y^{(n)}_k}
\partial_i\pdelim{u_k(\exp_{y^{(n)}_k}(\xi))}
\partial_j \pdelim{\Chi(\xi)\,w^{(n)}(2^{j^{(n)}_k}\xi)}
\sqrt{\det g_{i,j}^{y^{(n)}_k}(\xi)}{\mathrm d}\xi-
\nonumber
\\
\nonumber
2\sum_n2^{j_k^{(n)}r}\int_{|\xi|<\rho}u_k(\exp_{y^{(n)}_k}(\xi))
\Chi(\xi)\,w^{(n)}(2^{j^{(n)}_k}\xi)\sqrt{\det g_{i,j}^{y^{(n)}_k}(\xi)}{\mathrm d}\xi
=
\\
\|u_k\|^2-\|u\|^2+o(1)-
\nonumber
\\
2\sum_n\int_{|\eta|<\rho 2^{j_k^{(n)}}} \sum_{i,j=1}^N g_{ij}^{y^{(n)}_k}
\partial_i \pdelim{2^{-j_k^{(n)}r}u_k\circ \exp_{y^{(n)}_k}(2^{-j_k^{(n)}}\eta)}
\partial_j \pdelim{\Chi(2^{-j_k^{(n)}}\eta)\,w^{(n)}(\eta)}\cdot \quad
\nonumber
\\
\sqrt{\det g_{i,j}^{y^{(n)}_k}(2^{-j_k^{(n)}}\eta)}\;{\mathrm d}\eta -
\nonumber
\\
2\sum_n
2^{-2j_k^{(n)}} \int_{|\eta|<\rho 2^{j_k^{(n)}}}2^{-j_k^{(n)}r}u_k\circ\exp_{y^{(n)}_k}(2^{-j^{(n)}_k}\eta) \Chi(2^{-j^{(n)}_k}\eta)w^{(n)}(\eta)\sqrt{\det g_{i,j}^{y^{(n)}_k}(2^{-j^{(n)}_k}\eta)}
{\mathrm d}\eta=
\nonumber
\\
\|u_k\|^2-\|u\|^2+o(1)-
2\sum_n\int_{\Rset^N} \sum_{i}^N |\partial_i w^{(n)}(\eta)|^2{\mathrm d}\eta
-2\sum_n
2^{-2j_k^{(n)}} \int_{\Rset^N}|w^{(n)}(\eta)|^2{\mathrm d}\eta=
\nonumber
\\
\|u_k\|^2-\|u\|^2-2\sum_n \|\nabla w^{(n)}\|_2^2 +o(1).
\end{align*}
(In the third equality we have set $\eta=2^{j_k^{(n)}}\xi$, while in the fourth we have used the fact, due to \eqref{eq:profileM} that
$2^{-j_k^{(n)}} \Chi(2^{-j^{(n)}_k}\cdot) (u_k\circ \exp_{y^{(n)}_k})(2^{-j_k^{(n)}}\cdot){\rightharpoonup} \Chi(0) w^{(n)}= w^{(n)}$ as $k\to\infty$ (in our slightly modified sense of weak convergence). Note also we have still denoted by $\partial_i$ (resp. $\partial_j$) the derivative with respect to the $i^{th}$ (resp $j^{th}$) component of $\eta=2^{j_k^{(n)}}\xi$. Finally in the last equality we have used \eqref{eq:norm}).

In order to estimate the second line of \eqref{eq:preParsevalM} we shall split (according to \eqref{eq:norm}) the $H^{1,2}(M)$-norm into the $L^2$-norm of the gradient (gradient part) and the $L^2$-norm of the function ($L^2$ part) and consider first the latter.
Since
\begin{align*}
\sum_n \Vdelim{2^{j_k^{(n)}\frac{N-2}{2}}\,\Chi\circ\exp^{-1}_{y^{(n)}_k}\,w^{(n)}(2^{j^{(n)}_k}\exp^{-1}_{y^{(n)}_k}(\cdot))}_2^2=
\\
\sum_n 2^{j_k^{(n)}(N-2)}\int_{\mathcal{B}_\rho(y_n)} |\Chi\circ\exp^{-1}_{y^{(n)}_k}(x)\,w^{(n)}(2^{j^{(n)}_k}\exp^{-1}_{y^{(n)}_k}(x))|^2 {\mathrm d} v_g
=
\\
\sum_n 2^{j_k^{(n)}(N-2)}\int_{|\xi|<\rho}|\Chi(\xi)(w^{(n)}(2^{j^{(n)}_k}\xi)|^2\sqrt{\det g_{i,j}^{y^{(n)}_k}(\xi)}{\mathrm d}\xi
=
\\
\sum_n2^{-2j_k^{(n)}}\int_{|\eta|<\rho 2^{j_k^{(n)}}} |\Chi(2^{-j^{(n)}_k}\eta)w^{(n)}(\eta)|^2\sqrt{\det g_{i,j}^{y^{(n)}_k}(2^{-j^{(n)}_k}\eta)}
{\mathrm d}\eta{\rightarrow} 0\mbox{ as } k\to\infty,
\end{align*}
(since $j_k^{(n)}{\rightarrow}\infty$) as $k\to\infty$, the second line of \eqref{eq:preParsevalM} is evaluated in the limit by the sum of the gradient terms as follows.
\begin{equation*}
\begin{split}
\sum_n 2^{j_k^{(n)}(N-2)}
\int_{\mathcal{B}_\rho({y^{(n)}_k})}
\,\vdelim{{\mathrm d}\pdelim{\Chi\circ\exp^{-1}_{y^{(n)}_k}(x)\,w^{(n)}(2^{j^{(n)}_k}\exp^{-1}_{y^{(n)}_k}(x))}}^2
{\mathrm d} v_g=
\\
\sum_n 2^{j_k^{(n)}(N-2)}
\int_{|\xi|<\rho}\sum_{i,j=1}^N g_{ij}^{y^{(n)}_k}(\xi)
\partial_i
\pdelim{\Chi(\xi)\,w^{(n)}(2^{j^{(n)}_k}\xi)}
\partial_j
\pdelim{\Chi(\xi)\,w^{(n)}(2^{j^{(n)}_k}\xi)}
\sqrt{\det g_{i,j}^{y^{(n)}_k}(\xi)}{\mathrm d}\xi=
\\
\sum_n\int_{|\eta|<\rho 2^{j_k^{(n)}}} \sum_{i,j=1}^N g_{ij}^{y^{(n)}_k}
\partial_i
\pdelim{\Chi(2^{-j_k^{(n)}}\eta)\,w^{(n)}(\eta)}
\partial_j
\pdelim{\Chi(2^{-j_k^{(n)}}\eta)\,w^{(n)}(\eta)}
\sqrt{\det g_{i,j}^{y^{(n)}_k}(2^{-j_k^{(n)}}\eta)}\;{\mathrm d}\eta \quad
\\
{\rightarrow} \sum_n\int_{\Rset^N}|\nabla w^{(n)}(\eta)|^2\;{\mathrm d}\eta
= \sum_n\Vdelim{\nabla w^{(n)}}^2 \mbox{ as }{k\rightarrow \infty}.
\end{split}
\end{equation*} 
Consider now the terms in the sum in third line of  \eqref{eq:preParsevalM}.
Note that the $L^2$-part of the scalar product converges to zero by Cauchy inequality and by the calculations for the first line of \eqref{eq:preParsevalM}. 
At the light of the orthogonality condition \eqref{eq:expl_orth} we have to face two cases.

Case 1: The sequence $(j_k^{(n)}-j_k^{(m)})_{k\in\Nset}$ is unbounded. Assume without loss of generality that  $j_k^{(n)}-j_k^{(m)}{\rightarrow} +\infty$ as $k\rightarrow\infty$. Then, using changes of variables $\xi=\exp_{y^{(n)}_k}^{-1}(x)$ and $\eta=2^{j_k^{(n)}}\xi$,
\begin{equation*}
	\begin{split}
		\adelim{2^{j_k^{(m)} r}\,\Chi\circ\exp_{y^{(m)}_k}^{-1}(x)\,w^{(m)}\pdelim{2^{j^{(m)}_k}\exp^{-1}_{y^{(m)}_k}(\cdot)},2^{j_k^{(n)} r}\,\Chi\circ\exp_{y^{(n)}_k}^{-1}(x)\,w^{(n)}\pdelim{2^{j^{(n)}_k}\exp^{-1}_{y^{(n)}_k}(\cdot)}}=
		\\
		2^{j_k^{(n)} r}2^{j_k^{(m)}r}
		\int_{\mathcal{B}_\rho(y^{(m)}_k)\cap \mathcal{B}_\rho(y^{(n)}_k)}
		\,{\mathrm d}\pdelim{
			\Chi\circ\exp_{y^{(m)}_k}^{-1}(x)\,w^{(m)}\pdelim{2^{j^{(m)}_k}\exp^{-1}_{y^{(m)}_k}(x)}}\cdot\quad
		\\
		{\mathrm d}\pdelim{\Chi\circ\exp_{y^{(n)}_k}^{-1}(x)\,w^{(n)}\pdelim{2^{j^{(n)}_k}\exp^{-1}_{y^{(n)}_k}(x)}}
		{\mathrm d} v_g+o(1)=
		\\
		2^{j_k^{(n)} r}2^{j_k^{(m)} r}
		\int_{|\xi|<\rho}\sum_{i,j=1}^N g_{ij}^{y^{(n)}_k}(\xi)
		\partial_i \pdelim{\Chi(\xi)\,w^{(n)}(2^{j_k^{(n)}}\xi)}\cdot\quad
		\\
		\partial_j \pdelim{\Chi(\exp_{y_k^{(m)}}^{-1}(\exp_{y_k^{(n)}} (\xi)))\,w^{(m)}(2^{j^{(m)}_k}\exp_{y_k^{(m)}}^{-1}(\exp_{y_k^{(n)}} (\xi)))}
		\sqrt{\det g_{i,j}^{y^{(n)}_k}(\xi)}{\mathrm d}\xi=
		\\
		2^{-j_k^{(n)} r}2^{j_k^{(m)} r}
		\int_{|\eta|<\rho 2^{j_k^{(n)}}}\sum_{i,j=1}^N g_{ij}^{y^{(n)}_k}(2^{-j_k^{(n)}}\eta)
		\partial_i \pdelim{(1+o(1))\,w^{(n)}(\eta)}\cdot\quad
		\\
		\partial_j\pdelim{(1+o(1)) \,w^{(m)}(2^{j^{(m)}_k}\exp_{y_k^{(m)}}^{-1}(\exp_{y_k^{(n)}} (2^{-j_k^{(n)}}\eta)))}
		(1+o(1)){\mathrm d}\eta+o(1)\rightarrow 0,
	\end{split}
\end{equation*}
since, by \eqref{eq:profileM},
\begin{equation*}
\wlim_{k\rightarrow \infty}
2^{-j_k^{(n)} r}2^{j_k^{(m)}r}w^{(m)}(2^{j^{(m)}_k}(\exp_{y_k^{(m)}}^{-1}\circ\exp_{y_k^{(n)}}) (2^{-j_k^{(n)}}\cdot))=\wlim_{k\rightarrow \infty}2^{-j_k^{(n)} r} u_k(\cdot)=0.
\end{equation*}

Case 2: $2^{j_k^{(n)}}d(y_k^{(n)},y_k^{(m)})\to\infty$ as $k\to\infty$.  
Since case 1 has been ruled out, we can assume without restrictions that the sequence $j_k^{(m)}-j_k^{(n)}=j\in\Rset$ for all large $k$. Then, by arguing as above (and in particular by taking into account that the $L^2$-part of the scalar product is negligible), we get that, as $k\to\infty$,
\begin{equation*}
\adelim{2^{j_k^{(m)} r}\,\Chi\circ\exp^{-1}_{y^{(m)}_k}\,w^{(m)}(2^{j^{(m)}_k}\exp^{-1}_{y^{(m)}_k}(\cdot)),2^{j_k^{(n)} r}\,\Chi\circ\exp^{-1}_{y^{(n)}_k}\,w^{(n)}(2^{j^{(n)}_k}\exp^{-1}_{y^{(n)}_k}(\cdot))}\rightarrow 0,
\end{equation*}
since 
%
the values of $w^{(m)}$ and of $w^{(n)}$ are set to concentrate at sufficiently separated points, indeed
$d(2^{j_k^{(n)}}y_k^{(n)}, 2^{j_k^{(m)}}y_k^{(m)})=2^{j_k^{(n)}}
d(y_k^{(n)}, 2^{j}y_k^{(m)})\geq 
2^{j_k^{(n)}}
d(y_k^{(n)}, y_k^{(m)})\to\infty$. 

Then, by applying the estimates obtained for the three lines of inequality \eqref{eq:preParsevalM} we finally deduce \eqref{PlancherelM} concluding the proof of Theorem~\ref{thm:StSol}. 


\vskip5mm\noindent

\noindent{Acknowledgment.}
The first author is supported by GNAMPA of the ``Istituto Nazionale di
Alta Matematica (INdAM)'' and by MIUR - FFABR - 2017 research grant.

\noindent http://dx.doi.org/10.13039/501100003407
\vskip1mm\noindent
The second author had no academic affiliation when working on this paper.


%
  
\bibliographystyle{acmurl}

{\small






%


%
%

\end{document}